\documentclass[leqno,11pt,oneside]{amsart}

 \usepackage[english]{babel}

\usepackage{enumitem}
\usepackage{amsmath,amsfonts,amssymb,amsthm}

\usepackage{mathrsfs}

\usepackage[english]{babel}

\newtheorem*{theorem*}{Theorem}

\newtheorem{teo}{Theorem}[section]
\newtheorem{prop}[teo]{Proposition}

\newtheorem{example}[teo]{Example}
\newtheorem{cor}[teo]{Corollary}

\newtheorem{lem}[teo]{Lemma}
\newtheorem*{cor*}{Corollary}

\newtheorem*{lem*}{Lemma}
\newtheorem*{teorm}{Theorem 1}

\newtheorem*{teorA'}{Theorem A'}

\newtheorem*{fact*}{Fact}

\theoremstyle{definition}

\newcommand{\C}{\mathbb{C}}

\newcommand{\R}{\mathbb{R}}
\newcommand{\Pe}{\mathbb{P}}

\newcommand{\SC}{\mathscr{C}}

\newcommand{\F}{\mathcal{F}}

\newcommand{\G}{\mathcal{G}}
\newcommand{\LL}{{\mathcal L}}

\newcommand{\M}{\mathrm{M}}

\newcommand{\cl}[1]{\mbox{$\mathcal{#1}$}}

\newcommand{\codim}{{\rm codim}}

\newcommand{\IM}{\mathfrak{Im}}

\newcommand{\bb}{\mbox{\rm BB}}

\newcommand*\xbar[1]{ %
   \hbox{ %
     \vbox{%
       \hrule height 0.3pt 
       \kern0.35ex
       \hbox{%
         \kern-0.1em
         \ensuremath{#1}%
         \kern-0.1em
       }%
     }%
   }%
}

\newcommand*\xxbar[1]{%
   \hbox{%
     \vbox{%
       \hrule height 0.3pt 
       \kern0.4ex
       \hbox{%
         \kern-0.1em
         \ensuremath{#1}%
         \kern-0.1em
       }%
     }%
   }%
}

\DeclareMathOperator{\sing}{Sing}
\makeindex

\begin{document}

\title{Chow's theorem for real analytic Levi-flat hypersurfaces}

\author{Arturo Fern\'andez-P\'erez}
\address{Departamento de Matem\'atica ---  Universidade Federal de Minas Gerais}
\curraddr{Av. Ant\^onio Carlos 6627 --- 31270-901 --- Belo Horizonte, BRAZIL.}
\email{fernandez@ufmg.br}

\author{Rog\'erio   Mol}
\address{Departamento de Matem\'atica ---  Universidade Federal de Minas Gerais}
\curraddr{Av. Ant\^onio Carlos 6627 --- 31270-901 --- Belo Horizonte, BRAZIL.}
\email{rmol@ufmg.br}

\author{Rudy Rosas}
\address{Pontificia Universidad Cat\'olica del Per\'u}
\curraddr{Av. Universitaria 1801 --- Lima,
PERU.}
\email{rudy.rosas@pucp.edu.pe}


\thanks{First and second   authors  partially financed by Pronex-Faperj. First author
supported by a CNPq grant  PQ2019-302790/2019-5. Third author supported by
 Vicerrectorado de investigaci\'on de la Pontificia Universidad Cat\'olica del Per\'u}

\begin{abstract} In this article we provide a version of Chow's theorem for real analytic Levi-flat hypersurfaces in the complex projective space $\Pe^{n}$, $n \geq 2$.
More specifically,
we prove that a real analytic Levi-flat hypersurface  $M \subset \Pe^{n}$, with singular set of real dimension at most $2n-4$ and whose Levi leaves are contained in algebraic hypersurfaces, is tangent to the levels of a rational function in $\Pe^{n}$. As a consequence,  $M$ is a semialgebraic set.
We also prove that a Levi foliation on $\Pe^{n}$ ---  a singular real analytic foliation whose leaves are immersed  complex manifolds of codimension one --- satisfying similar conditions ---
singular set of real dimension at most $2n-4$ and all leaves algebraic --- is defined by the level sets of a rational function.
\end{abstract}

\maketitle

\section{Introduction}
Chow's Theorem is an emblematic result in complex geometry, providing a   link between analytic geometry and algebraic geometry.
It asserts that every closed  analytic set in the complex projective space $\mathbb{P}^{n}$ is an algebraic set.
Originally proved in 1949 by W.-L. Chow \cite{chow1949}, in the following  years   it gained new proofs, based on different techniques, in 1953, by
R. Remmert and K. Stein  \cite{remmert1953} 
 and in 1956 by J.-P. Serre \cite{serre1956}.  
Chow's Theorem is complex in its essence and, in its full generality,   is not  extendable to the real framework.
There are many examples of closed real analytic sets in the complex projective space $\mathbb{P}^{n}$  that are not real algebraic (nor semialgebraic).

One attempt to find formulations similar to Chow's Theorem for  real analytic objects in the complex space $\Pe^{n}$ is to work with varieties having some sort of
``complex character''. This is an attribute, for instance, of  Levi-flat hypersurfaces.
We recall that a real analytic subvariety $M$ of real codimension one (a \emph{hypersurface}, in our terminology) in an ambient  complex manifold of dimension $n$ is \emph{Levi-flat} if its
regular part $M_{reg}$ --- the set of points  near which $M$ is a submanifold of real dimension $2n-1$ --- has a real analytic foliation $\LL$ by complex leaves of dimension $n-1$. This foliation is called \emph{Levi foliation} and its leaves are called \emph{Levi leaves}. In  Section \ref{section-levi-flat}  we give definitions and some properties of these objects.

Nevertheless, even in the Levi-flat context, a formulation such as Chow's Theorem fails to be true in general.
First, we call attention to the fact that if a Levi-flat hypersurface is real algebraic (or semialgebric), then all Levi leaves are algebraic, meaning that their closures are complex algebraic hypersurfaces. Indeed, they are contained in algebraic Segre varieties (\cite[Cor. 2.3]{jiri2012}; see also   Section \ref{section-levi-flat}).
Thus, asking $\LL$ to have algebraic leaves is a natural hypothesis in an attempt to obtain a Levi-flat type of Chow's Theorem.
However, this hypothesis alone is not sufficient: J. Lebl in \cite[Sec. 5]{lebl2015} gives an example of a real analytic singular Levi-flat hypersurface in $\mathbb{P}^{2}$ whose Levi leaves are algebraic   but is not
semialgebraic (we describe it in Example \ref{ex-lebl} below).
Second, it is a consequence of Darboux-Jouanolou's Theorem \cite[Th. 3.3]{jouanolou1979} that, if the Levi foliation $\LL$ contains infinitely many algebraic leaves and extends to a singular holomorphic foliation of codimension one $\F$ in the ambient space $\Pe^{n}$, then  $\F$ admits a rational first integral, i.e. its
leaves   are contained in the levels of a rational function. This is  the key tool for   the Levi-flat Chow's Theorem given in \cite[Th. 1.1]{jiri2012}, where, aside from asking infinitely many algebraic Levi leaves,
local  hypotheses are set which ultimately lead to the extension of $\LL$  to a singular holomorphic foliation on  $\Pe^{n}$.
Evidently, in the afore cited Lebl's example, such an extension is impossible.
With these comments in mind, we can state a Chow type theorem for Levi-flat hypersurfaces, which is the main result of this article:

\begin{teorm}
\label{teo-extension-ndim}
Let $M$ be a real analytic Levi-flat hypersurface  in
$\mathbb{P}^n$ such that $\sing(\xbar{M_{reg}})$ has real dimension at most $2n-4$.  Suppose that the Levi leaves
of $M$ are all algebraic. Then the Levi foliation extends to a singular holomorphic
 foliation on $\mathbb{P}^n$ with a rational first integral. As a consequence, $M$ is semialgebraic.
\end{teorm}
The theorem's statement makes reference to  the semianalytic set $\xbar{M_{reg}} \subset M$,   the closure taken in the usual topology.
Components of ``umbrella handle'' type may occur in $M$, so  this inclusion can be proper.
The  extension  of the Levi foliation to the ambient $\Pe^{n}$, in the theorem's conclusion,    is proved through purely algebraic arguments, using
B\'ezout's Theorem. It
 does not depend on additional hypothesis on the local geometry of $M$ in its singularities, as done in \cite[Th. 1.1]{jiri2012}.
The condition on the dimension of $\sing(\xbar{M_{reg}})$  seems to be reasonable, taking into account that   our ultimate  goal is to prove that $M$ is tangent to
a singular holomorphic foliation $\F$ in the ambient $\Pe^{n}$. In this case, $\dim_{\R} \sing(\xbar{M_{reg}}) \leq 2n - 4$ means, roughly
speaking, that the singularities of $\xbar{M_{reg}}$ comes from those of $\F$ (see Lemma \ref{lem-dim-sing} below).

The article has the following structure.
In Section \ref{section-levi-flat}, we present definitions and general properties concerning Levi-flat hypersurfaces and Segre varieties.
We also prove some  results of general character concerning the dimension of the singular set of a Levi-flat hypersurface invariant by an ambient singular holomorphic
foliation. For instance, a Levi-flat hypersurface in $\Pe^{n}$, for $n \geq 3$, tangent to a globally defined singular holomorphic foliation
$\F$, contains a component of complex codimension two of the singular set of $\F$ (Proposition \ref{prop-existence-codimtwo}).
We also prove results concerning Segre degenenerate points (Proposition \ref{prop-segredegenerate-codimtwo}) and ``umbrella handle'' type points (Propostion
\ref{prop-stick-handle}) of Levi-flat hypersurfaces invariant by ambient holomorphic foliations.
Next, in Section \ref{section-chow}, which   certainly is the core of this article, Theorem \ref{teo-extension-ndim} is proved.  The result is first proved for $\Pe^{2}$, as
a  direct consequence of the following fact: an infinite family of algebraic curves of the same
degree in $\Pe^{2}$, whose singularities  and pairwise intersection points  form a discrete set, contains infinitely many curves that lie in a pencil (Proposition \ref{prop-infinite-collection}).
Theorem \ref{teo-extension-ndim}, for an arbitrary $\Pe^{n}$ with $n \geq 3$,
is then obtained from the result in $\Pe^{2}$ by taking generic two dimensional complex linear sections.
The reading of Section \ref{section-chow} is independent of the other sections of this paper.
Finally, in Section \ref{section-Levi-flat-foliations}, we give an application of the method leading to   Theorem \ref{teo-extension-ndim} for the study of singular  Levi foliations on $\Pe^{n}$.
We say that a real analytic foliation of real codimension two on a complex manifold is a \emph{Levi foliation} if its leaves are immersed holomorphic manifolds
of complex codimension one. It is a \emph{singular Levi foliation} if it is  locally   defined by real analytic $1$-forms of type (1,0).
We close this article by proving, in  Theorem \ref{teo-levi-foliation}, that
 a  singular Levi foliation on $\Pe^{n}$, $n \geq 2$, must have   a rational first integral, provided it has algebraic leaves and
 singular set has real dimension at most $2n-4$.

\section{Singular Levi-flat hypersurfaces and Segre varieties}
\label{section-levi-flat}

The goal of this section is to present some  general facts  concerning the structure and  dimension of singular sets of   Levi-flat hypersurfaces invariant by singular  holomorphic foliations in the ambient complex manifold, in particular when this manifold is $\Pe^{n}$. Some of these facts  have already appeared in the literature of the field in a scattered manner. Our intention is to give them a more systematic presentation.

Let $X$ be a complex manifold of $\dim_{\C} X = n \geq 2$.
Recall that a singular holomorphic foliation of codimension one $\F$ on  $X$ is the object given by a (regular) holomorphic foliation of codimension one outside a
complex analytic subset $\sing(\F)$ --- its singular set --- of codimension at least two. Locally, $\F$ is defined, up to multiplication by   non-vanishing holomorphic functions,  by a holomorphic $1$-form, say $\omega$, with $\sing(\omega)$ of codimension at least two, satisfying
the integrability condition ($\omega \wedge d \omega = 0$). In this case,  $\sing(\F)$ coincides locally with $\sing(\omega)$.
We would refer the reader to the book \cite{linsneto2020} for a thorough treatment on singular holomorphic foliations, especially those ambiented
in complex projective spaces, which are object of our interest.
Hereafter, whenever we mention a
singular holomorphic foliation it will be implicit that its codimension is one.

Let $M \subset X$ be an irreducible real analytic variety of real codimension one, i.e. $\dim_{\R} M = 2n -1$.
Throughout this text we use the terminology \emph{real analytic hypersurface} when referring to such an $M$.
The \emph{regular part} of $M$ is the semianalytic set formed by the points $p \in M$ for which there exists a neighborhood $p \in U \subset X$ such that $M \cap U$ is a real analytic manifold
of real dimension $2n-1$. We denote this set by $M_{reg}$.
Let us consider $\xbar{M_{reg}}$, where the bar stands for the closure in the topology of $X$.
 We denote
$\sing(\xbar{M_{reg}}) = \xbar{M_{reg}} \setminus M_{reg}$.
If $p \in M_{reg}$, let $T_{p}M$ denote the tangent space of $M$ and $T_{p}X$, the complex tangent space of $X$, which also has a structure of an $\R$-vector space of dimension $2n$.
Consider the
 canonical inclusion of $\R$-vector spaces $T_{p}M \hookrightarrow T_{p}X$.
It is a standard fact from linear algebra that there is a unique complex subspace $\LL_{p} \subset T_{p}M$ such that
 $\dim_{\C} \LL_{p} = n-1$. The correspondence $p \in M_{reg} \to \LL_{p}$ defines a real analytic distribution  of complex hyperplanes. If this distribution is
 integrable in the sense of Frobenius, then it defines a real analytic foliation on $M_{reg}$, still denoted by $\LL$, whose leaves are immersed complex manifolds of  complex dimension $n-1$. We say in this case that $M$ is \emph{Levi-flat} and that $\LL$ is its \emph{Levi foliation}.
If $p \in M_{reg}$, take $z=(z_{1},\ldots,z_{n})$ holomorphic coordinates for $X$ in a neighborhood $U$ of $p$ and a real analytic function $\varphi = \varphi(z,\bar{z})$
  such that $M \cap U = \{ \varphi = 0\}$ and $\nabla \varphi (z) \neq 0$ for every $z$. Then the Hermitian quadratic form
  \[ L_{z}(v) = \sum_{1 \leq i,j \leq n} \frac{\partial^{2} \varphi}{\partial z_{i} \partial \bar{z}_{j}}(z,\bar{z}) v_{i} \bar{v}_{j}, \ \ v = (v_{1},\ldots,v_{n}) \in \C^{n} \]
is called \emph{Levi-form}. $M$ is Levi-flat on $U$ if and only if  $L_{z} \equiv 0$. Since $M$ is connected, this local
condition, around one fixed $p \in M_{reg}$, is enough to assure that the whole $M$ is Levi-flat.

A real analytic hypersurface $M \subset X$ is Levi-flat if and only if  Cartan's local normal form \cite{cartan1933}  applies:
  each point of $M_{reg}$ has a neighborhood in $X$ where there are holomorphic coordinates $z = (z_{1}, \ldots, z_{n})$ in which
$M_{reg}$ corresponds to $\IM(z_{n}) = 0$ and $\LL$ to the foliation whose leaves are $z_{n} = c$, where $c \in \R$. Given this local normal form, the holomorphic foliation with leaves
$z_{n}=c$, for $c \in \C$, extends $\LL$ to a neighborhood of $p \in M_{reg}$ in $X$. It is a trivial fact that this local extension is unique.  Therefore, by gluing
together  local extensions, we can
obtain  a holomorphic foliation, defined in a whole neighborhood of $M_{reg}$ in $X$, that extends $\LL$.
However, in general, for $p \in \sing(\xbar{M_{reg}})$,  the foliation $\LL$ does not extend to a singular holomorphic foliation in  any neighborhood of $p$
(see Example \ref{ex-brunella} below, where the Levi foliation extends as a 2-web in the ambient space).
 We should mention Cerveau-Lins Neto's Theorem
 \cite{cerveau2011}, which asserts that, if $\F$ is a singular holomorphic foliation  extending $\LL$ in a   neighborhood of $p \in \sing(\xbar{M_{reg}})$ in $X$, then there is a meromorphic
(possibly holomorphic) function
$F$ defined in a perhaps smaller neighborhood $U$ of $p$ such that the leaves of $\F$ in $U$ are contained in the level sets of $F$.  We say
that $F$ is a \emph{meromorphic}
(or  \emph{holomorphic}) \emph{first integral} for $\F$.

Actually, if we have an extension of $\LL$,  as a singular holomorphic foliation, around each point of $\sing(\xbar{M_{reg}})$,
the uniqueness of the local extensions gives
  a singular holomorphic foliation $\F$, defined in
  a   neighborhood of $\xbar{M_{reg}}$.   We say in this case that $M$ in \emph{invariant} by
  or  \emph{tangent} to $\F$.
 When $X = \Pe^{n}$, the connected components of the complementary of
$\xbar{M_{reg}}$ are Stein manifolds \cite{fernandez2017}. Thus, this foliation $\F$
  can be further extended   to a singular holomorphic foliation  in  the ambient  $\Pe^{n}$
 \cite[Lem. 2]{linsneto1999}.
In these circumstances,  we can take advantage of the fact that singular holomorphic foliations  in $\Pe^{n}$ are relatively well understood objects in order to obtain geometric information
on $M$ itself.

If $\F$ is a singular holomorphic foliation on a complex manifold $X$, then the singular set $\sing(\F)$ is a complex analytic variety
such that $\codim_{\C}\sing(\F) \geq 2$.
The following simple facts concern the dimension of the singular set of a real analytic Levi-flat hypersurface tangent to a singular holomorphic foliation:

\begin{lem}
\label{lem-dim-sing}
 Let $M$ be a real analytic Levi-flat hypersurface contained in a complex manifold $X$ of dimension $n$, tangent to a singular holomorphic foliation $\F$ in $X$.
Then, $\sing(\F) \cap \xbar{M_{reg}} \subset \sing(\xbar{M_{reg}})$.
If this inclusion is proper, then $\dim_{\R} \sing(\xbar{M_{reg}}) = 2n-2$ and the components of $\sing(\xbar{M_{reg}})$ of maximal  dimension are complex analytic and invariant by $\F$. Besides, $\sing(\xbar{M_{reg}})$ does not have components of real dimension $2n-3$.
\end{lem}
\begin{proof}
The first assertion  follows from Cartan's normal form, which implies that a regular point of $M$ must be non-singular  for $\F$.
For the  assertions concerning the proper inclusion, if $p \in \sing(\xbar{M_{reg}}) \setminus \sing(\F)$, it is enough to take  holomorphic coordinates that trivialize
$\F$ around $p$ in order to see that the leaf of $\F$ passing through $p$ must be contained in $\sing(\xbar{M_{reg}})$.
\end{proof}

\begin{example}
{\rm
It is a simple task    to produce examples of Levi-flat hypersurfaces such that  $\dim_{\R}\sing(\xbar{M_{reg}}) = 2n-2$, where $n$ is the complex dimension of the ambient
space. For instance, take    a real analytic curve $S \subset \C$ with a singularity at $0 \in \C$ and consider $M = S \times \C^{n-1} \subset\C^{n}$.
Then $M$ is Levi-flat, tangent to the non-singular vertical foliation whose leaves are $\{z\} \times \C^{n-1}$, with $z \in \C$. In this case,
  $\sing(\xbar{M_{reg}})= \sing(M) = \{0\} \times \C^{n-1}$.

This construction can be carried out  in the complex projective space
 $\Pe^n$ by taking  $S \subset \C$    real algebraic, with a singularity at $0 \in \C$, defined by a polynomial equation $p(x,y)=0$, where $z=x+iy$.
Take  homogeneous coordinates $[z_0:z_1:\ldots:z_n]$  in $\Pe^n$,
 consider the affine hyperplane $\{z_0 \neq 0\} \cong \C^{n}$ and set
 $z=z_1/z_0$ as an affine coordinate. Define, as above,
  $M = S \times \C^{n-1} \subset \C^{n}$.
The Zariski closure of $M$ is
 a real algebraic Levi-flat hypersurface   in $\Pe^{n}$, still denoted by $M$,
whose equation can be obtained by
 bihomogenizing $p(x,y)$. This
 $M$ is   tangent to the holomorphic foliation $\F$ induced by $\omega=z_0 dz_1-z_1dz_0$, whose singular set is the $(n-2)$-dimensional projective plane
 $\{z_{0} = z_{1} = 0 \}$. On the other hand,  $\sing(\xbar{M_{reg}})$ has real dimension $2n-2$, since it contains $\{z_{1} = 0 \}$.
}\end{example}

If $\F$ is a singular holomorphic foliation on
 $\Pe^{n}$, where  $n \geq 3$, then $\sing(\F)$ contains irreducible components of complex codimension two \cite{jouanolou1979}.
Therefore, if $M \subset \Pe^{n}$ is real analytic Levi-flat hypersurface tangent to $\F$,  necessarily $\sing(\F) \cap \xbar{M_{reg}} \neq \emptyset$,
seeing that
 $\Pe^{n} \setminus \xbar{M_{reg}}$ has Stein connected components.
In particular, this implies that, for $n \geq 3$, there are no smooth real analytic Levi-flat hypersurfaces in $\Pe^{n}$ \cite{linsneto1999}.
Indeed, by the above discussion, the Levi foliation of a hypothetical   smooth Levi-flat hypersurface would extend to a singular holomorphic foliation
on $\Pe^{n}$ and, by Lemma \ref{lem-dim-sing}, this is incompatible with $\sing(\F) \cap \xbar{M_{reg}} \neq \emptyset$.

When $n \geq3$, we prove the following fact concerning the singular set of a Levi-flat hypersurface invariant by a globally
defined singular holomorphic foliation on   $\Pe^{n}$:

\begin{prop}
\label{prop-existence-codimtwo}
 Let $M \subset \Pe^{n}$ be a real analytic Levi-flat hypersurface, where $n \geq 3$, invariant by a singular holomorphic foliation $\F$ of codimension one
in the ambient $\Pe^{n}$.
Then there exists a component of complex codimension two of $\sing(\F)$ contained in  $\xbar{M_{reg}}$.
\end{prop}
Remark that, under the hypothesis of the proposition, Cerveau-Lins Neto's theorem applies, providing a meromorphic first integral for the ambient foliation
in a small neighborhood of each point of $\xbar{M_{reg}}$. Taking this into account,
the proof of the   proposition derives from the following Lemma:

\begin{lem}
\label{lem-existence-meromorphic}
In the same setting   of   Proposition \ref{prop-existence-codimtwo},
  there exists $p \in \sing(\xbar{M_{reg}})$ near which $\F$ has a purelly meromorphic first integral.
\end{lem}
\begin{proof}
We suppose, by contradiction, that near each  point of $\sing(\xbar{M_{reg}})$, the foliation $\F$ has a holomorphic first integral.
Suppose that $\F$ is induced, on the complex cone $\C^{n+1} \setminus \{0\}$ of $\Pe^{n}$, by an integrable polynomial $1$-form $\omega$ with homogeneous coefficients
 such that $\codim_{\C} \sing(\omega) = 2$.
A linear immersion $i: \Pe^{2} \hookrightarrow \Pe^{n}$ is said to be tranversal to $\F$ if
$ \sing(i^{*} \omega)$ is discrete,  where   $i$ still denotes  its lifting to the complex cones of $\Pe^{2}$ and $\Pe^{n}$. In this case, we denote by $i^{*} \F$
the foliation  induced by $i^{*} \omega$ in   $\Pe^{2}$.
  We note that linear immersions transversal to foliations are generic in the corresponding Grassmannian \cite{mattei1980}.
  The singular set $\sing(i^{*}\F)$   has three kinds of points:
\begin{enumerate}
\item points of intersection of $i(\Pe^{2})$ with  points of $\sing(\F)$ near which
$\codim_{\C} \sing(\F) > 2$;
\item points of intersection of $i(\Pe^{2})$ with the   components of
codimension two of $\sing(\F)$;
\item points of tangency of $i(\Pe^{2})$
 and the regular part of $\F$.
 \end{enumerate}
 The first set can be avoided by taking a generic immersion.
Anyway, near each such point, $\F$ has a holomorphic first integral by Malgrange's theorem \cite{malgrange1976}, and so does $i^{*}\F$.
Near each point in the third set, $\F$ is regular and so admits a holomorphic first integral, the same holding for $i^{*}\F$.
Thus,  for $p \in \sing(i^{*}\F)$ fitting these two cases,    $\bb_{p}(i^{*}\F) = 0$, where
$\bb$ stands for the Baum-Bott index. For the definition and properties of the Baum-Bott index used here we refer the reader to \cite[Sec. 3.1]{cerveau2013}.
By \cite[Th. 3.3]{cerveau2013}, if $\Gamma$ is a component of complex codimension two of $\sing(\F)$, then, for $p \in \Gamma$ outside a proper analytic set $\Gamma_{1} \subset \Gamma$,
the Baum-Bott index $\bb_{0}(j^{*}\F)$, defined by means of a local transversal immersion $j:(\C^{2},0) \to (\Pe^n,p)$,  does not depend on $p$ and on $j$.
Note that, using again that  $\Pe^{n} \setminus \xbar{M_{reg}}$ has Stein connected components, $\Gamma$ intersects $\xbar{M_{reg}}$. Since $\F$ admits a holomorphic
first integral near each point of $\xbar{M_{reg}}$,
we conclude that these indices must vanish. Thus, for $i: \Pe^{2} \to \Pe^{n}$ sufficiently generic, $\bb_{p}(i^{*}\F) = 0$ also  for points $p$ of the second of
the above types. Summarizing, if $i$ is sufficiently generic, we have that $\bb_{p}(i^{*}\F) = 0$ for every $p \in \sing(i^{*}\F)$. However,
if $d \geq 0$ is the degree of $i^{*}\F$ (i.e. the number of tangencies of $i^{*}\F$ with a generic line of $\Pe^{2}$, counted with multiplicities),
we  must have, by Baum-Bott's theorem (\cite{baum1972} and also \cite[Th. 3.1]{cerveau2013}), that
\[\sum_{p \in \sing(i^{*}\F)} \bb_{p}(i^{*}\F) = (d+2)^{2} >0 .\]
This gives a contradiction with the vanishing of all Baum-Bott indices.
\end{proof}

We can now prove the proposition:

\begin{proof}[Proof {\rm(of Proposition \ref{prop-existence-codimtwo})}]
 By Lemma \ref{lem-existence-meromorphic}, there exists $p \in \xbar{M_{reg}}$ around which $\F$ has a purely meromorphic first integral, say $F$.
Thus, we can find   two    leaves   of the Levi foliation  $\LL$, say $L_{1}$ and  $L_{2}$, contained in different levels of $F$,  accumulating to $p$. The intersection
of their closures   is an analytic set of complex codimension two, contained in the indeterminacy set of $F$,
also contained in $\sing(\F)$ and in $\xbar{M_{reg}}$. This proves the proposition.
  \end{proof}

We record  the following evident consequence of Lemma \ref{lem-dim-sing} and of Proposition \ref{prop-existence-codimtwo}:
\begin{cor}
\label{cor-singular-dim}
If $M \subset \Pe^{n}$ is a real analytic Levi-flat hypersurface, with $n \geq 3$, invariant by a singular holomorphic foliation   of codimension one
in   $\Pe^{n}$, then either
$\dim_{\R} \sing(\xbar{M_{reg}}) = 2n -4$ or $\dim_{\R} \sing(\xbar{M_{reg}}) = 2n -2$.
\end{cor}

In the case $n =2$, the validity of Corollary \ref{cor-singular-dim} would imply that there are no smooth real analytic Levi-flat
hypersurfaces in $\Pe^{2}$.
Such an object, if it existed, would be an exceptional minimal set for the singular holomorphic foliation obtained by the extension of its Levi foliation to the ambient $\Pe^{2}$.
We recall that an \emph{exceptional minimal set} for a singular holomorphic foliation of codimension one $\F$ in a complex manifold $X$ is a non-empty compact set $\mathcal{M} \subset X$ such that
$\mathcal{M} \cap \sing(\F) = \emptyset$, which is
invariant by $\F$ (i.e. if a leaf $L$ of $\F$ intersects $\mathcal{M}$ then $L \subset \mathcal{M}$) and is minimal with respect to these properties (see \cite{camacho1988}).
Whether or not there exists a foliation on $\Pe^{2}$ admitting  an exceptional minimal set is one of the most intriguing problems in the theory of holomorphic foliations.
We should mention   that, for $n \geq 3$, there are no   exceptional minimal sets  for singular holomorphic foliations in $\Pe^{n}$
\cite{linsneto1999}.

Let $M$ be an irreducible germ of real analytic hypersurface at $0 \in \C^{n}$. Suppose that  $\varphi:U \to \R$ is a real analytic function,
which is a defining function for $M$,
where
$U \subset \C^{n}$ is a neighborhood of the origin, small enough so that
$\varphi = \varphi(z,\bar{z})$ admits a Taylor series development in $U$, say
$\varphi(z,\bar{z}) = \sum_{\mu,\nu \geq 0} a_{\mu \nu}z^{\mu} \bar{z}^{\nu}$.
Note that, since   $\varphi$ is real, $a_{\mu \nu} = \bar{a}_{\nu \mu}$ for all pair of coefficients $\mu,\nu$.
In a more precise way,
considering  $U^{*} = \{z \in U;  \bar{z}= (\bar{z}_{1}, \ldots, \bar{z}_{n}) \in U\}$ and
setting  $w = \bar{z}$ as a new complex variable,
  the series
$\varphi_{\C}(z,w) = \sum_{\mu,\nu \geq 0} a_{\mu \nu}z^{\mu} w^{\nu}$ converges in $U \times U^{*} \subset \C^{n} \times \C^{n}$, defining a
holomorphic map $\varphi_{\C}= \varphi_{\C}(z,w)$,   the \emph{complexification} of $\varphi$. We then define the \emph{complexification} of $M$ as the complex hypersurface $M_{\C}$ in
$U \times U^{*}$ with equation $\varphi_{\C} = 0$.
By possibly reducing $U$, the defining function $\varphi$ can be obtained in such a way that
$\varphi_{\C}$ is a \emph{minimal} defining function for $M_{\C}$ (see \cite[Lem. 2.1]{Pinchuk_2017}). This means that if $W \subset U \times U^{*}$ is an open set and $f$
is a holomorphic function such that $f \equiv 0$ on $M_{\C} \cap W$, then there exists a holomorphic function $h$ on $W$ satisfying
  $f = h \varphi_{\C}$ on $W$. We also call $\varphi$ a \emph{minimal defining function} for $M$.

Suppose now that $M \subset X$ is a real analytic Levi-flat hypersurface in a complex manifold $X$ of $\dim_{\C}X=n$.
If $p \in M$,   we can choose a coordinate neighborhood $U \subset X$ around $p$,
with coordinates $z = (z_{1},\ldots,z_{n})$,
 and,  identifying each $q \in U$ with its $z$-coordinate,
a minimal defining function  $\varphi:U \to \R$    for
 $M \cap U$.
For $q \in M \cap U$, we call the complex analytic variety
$\Sigma_{q}^{U} = \{z \in U; \varphi(z,\bar{q}) = 0 \}$  the \emph{Segre variety} of $M$ at $q$. At this point, we refer the reader to \cite{Pinchuk_2017}, and also to
\cite{lebl2013,fernandezlebl2015,shafikov2015}, for details on the construction and properties of Segre varieties.
In principle, the Segre variety at $q$ depends on the neighborhood $U$ and on the defining function $\varphi$. However, its germ at $q$
is intrinsically defined.
It is evident that $q \in \Sigma_{q}^{U}$.
Thus,  $\Sigma_{q}^{U}$ is either a complex analytic hypersurface or, when $\varphi(z,\bar{q}) \equiv 0$, the whole $U$.
In the latter case, we say that $q$ is a \emph{Segre degenerate} point of $M$.
In the non-degenerate case, all components of $\Sigma_{q}^{U}$  contained in $M$ are invariant by the Levi-folition $\LL$.
In particular, if $q \in \M_{reg} \cap U$, then $\Sigma_{q}^{U}$ has a unique local component in $M$, which coincides with the Levi leaf through $q$.
For our purposes, this is the most relevant property of Segre varieties.
As a consequence,   all Levi leaves are closed in $M_{reg} \cap U$ (their closures in $U$ are complex hypersurfaces).
Since $\varphi$ is a real function, it is easy to see that, for $p,q \in M \cap U$,
$q \in \Sigma_{p}^{U}$ if and only if $p \in \Sigma_{q}^{U}$. This means, for instance, that if $p$ is Segre degenerate, then
$p \in \Sigma_{q}^{U}$ for every $q \in M_{reg} \cap U$, which means that $p$ belongs to the closure of every Levi leaf in $U$.
In particular, a Segre degenerate point is in $\sing(\xbar{M_{reg}})$.
In fact, it is   a \emph{dicritical singularity} of $M$, meaning that it lies in the closure
of infinitely many leaves of $\LL$.
It turns out that dicritical singularities are also Segre degenerate, as proved in \cite[Th. 3.1]{Pinchuk_2017}.

Segre degenerate points are contained in a complex analytic variety in $X$
of complex dimension at most $n-2$ (see, for instance
\cite[Prop. 3]{lebl2013}).
The following result is a consequence Proposition \ref{prop-existence-codimtwo}:

\begin{prop}
\label{prop-segredegenerate-codimtwo}
 Let $M \subset \Pe^{n}$ be a real analytic Levi-flat hypersurface, where $n \geq 3$, invariant by a singular holomorphic foliation $\F$ of codimension one
in  $\Pe^{n}$.
Then the set of  Segre degenerate points of $M$ form a complex algebraic variety, contained in $\sing(\xbar{M_{reg}})$,  of pure complex dimension $n-2$.
\end{prop}
\begin{proof} Segre degenerate points are dicritical. Thus,  their set coincide with
subset  of $\sing(\xbar{M_{reg}})$ formed by the indeterminacy points of
  the local meromorphic first integrals of  $\F$. By the proof of Proposition \ref{prop-existence-codimtwo}, this latter  set has local components of
dimension $n-2$ contained in $\sing(\xbar{M_{reg}})$. The complex analytic variety formed by assembling together these components is algebraic, as a consequence of
the authentic  Chow's Theorem.
Finally,   by Malgrange's theorem \cite{malgrange1976}, at a point $p \in \sing(\F)$ where the local dimension of $\sing(\F)$ is  less than
$n-2$, $\F$ admits a holomorphic first integral. If such a  $p$ lies in $\xbar{M_{reg}}$, then $p$ is non-dicritical and, thus, cannot be Segre degenerate
by \cite[Th. 3.1]{Pinchuk_2017}.
\end{proof}

Up to now, we have considered  properties of  $\xbar{M_{reg}}$, where $M$ is a real analytic Levi-flat hypersurface.
We finish this section by examining
the  ``handle stick'' points of $M$, that is to say, the points in $M \setminus \xbar{M_{reg}}$. When $M$ is tangent to an ambient singular holomorphic foliation
it turns out that $M \setminus \xbar{M_{reg}}$ is also invariant by the foliation.

\begin{prop}
\label{prop-stick-handle}
 Let $M$ be a real analytic Levi-flat hypersurface contained in a complex manifold $X$ of dimension $n$, tangent to a singular holomorphic foliation $\F$ in $X$.
Then, $M \setminus \xbar{M_{reg}}$ is invariant by $\F$.
\end{prop}
\begin{proof}
The problem is local, thus we can suppose that $M$ is defined in a neighborhood $U$ of $0 \in \C^{n}$ by some real analytic function 
having a Taylor series expansion in $U$.
Let $\F^{*}$ denote the singular holomorphic foliation on $U^{*}$ defined in the following way:
$\sing(\F^{*}) = (\sing(\F))^{*} = \{ \bar{p}; p \in \sing(\F)\}$ and,
if $L_{q}$ denotes the leaf of $\F$ through $q \in U \setminus \sing(\F)$, then  $L_{\bar{q}}^{*} = \{\bar{z}; z \in L_{q}\}$ is the leaf of $\F^{*}$
through $\bar{q} \in U^{*} \setminus \sing(\F^{*})$.
We consider the complexification $M_{\C} \subset U \times U^{*} \subset \C^{n} \times \C^{n}$.
The real analytic immersion $\kappa:z \in U \mapsto (z,\bar{z}) \in U \times U^{*}$
identifies $M$ is  with the real trace
$\kappa(M) = M_{\C} \cap \Delta$, where $\Delta = \{ (z,w) \in U \times U^{*}; w = \bar{z} \}$.
Let $\F \times \F^{*}$ denote the codimension two singular holomorphic foliation on $U \times U^{*}$, whose leaf trough
$(p,\bar{q}) \in U \times U^{*}$, where $p\not\in \sing(\F)$ and $\bar{q} \not\in \sing(\F^{*})$, is $L_{p} \times L_{\bar{q}}^{*}$
(note that if $U$ is small enough and $F(z)$ is a local meromorphic first integral for $\F$  in $U$,
then $\F \times \F^{*}$ is defined
by the levels of the map with meromorphic entries $(F(z),F^{*}(w))$, where $(z,w) \in U \times U^{*}$).
$M$ being invariant by $\F$ implies that $M_{\C}$ is invariant by $\F \times \F^{*}$ (see the proof of \cite[Lem. 7]{brunella2011}).
Let $p \in M \setminus \xbar{M_{reg}}$ be a ``handle stick'' point. Suppose that $p \not\in \sing(\F)$ and
consider $(p,\bar{p}) \in M_{\R}$.
First note that $\kappa(L_{p}) = \{(z,\bar{z}); z \in L_{p}\} \subset L_{p} \times L_{\bar{p}}^{*}$.
On the other hand,
 the invariance of $M_{\C}$ by $\F \times \F^{*}$ gives  that  $L_{p} \times L_{\bar{p}}^{*} \subset M_{\C}$.
Thus
$\kappa(L_{p})  \subset M_{\C} \cap \Delta$, implying that
$L_{p} \subset M$ and proving the proposition.
\end{proof}
Transferred to $\Pe^{n}$, the above result says that, when $M$ is tangent to a globally defined holomorphic foliation $\F$, then $M \setminus M_{reg}$ is a finite union of algebraic varieties formed by invariant algebraic hypersurfaces and by components of the singular set of $\F$.

\section{Chow's Theorem}
\label{section-chow}

This section is devoted to the proof of Theorem \ref{teo-extension-ndim}. This proof, carried out in dimension two and then extended to arbitrary dimensions
by taking two dimensional linear sections, is based upon
the use of B\'ezout's Theorem for collections of irreducible algebraic curves in the projective plane $\Pe^{2}$.
If $\SC$ is  such a collection, then a point $p\in\mathbb{P}^2$ is a \emph{singularity} of $\SC$ if either
$p$ is a singularity of some curve in $\SC$ or if $p$ lies in the intersection of two different curves in $\SC$.
Denote by $\sing(\SC)$ the set of these singularities.
We call attention to the following trivial  geometric fact: suppose that $\SC$ is an infinite collection of lines in $\mathbb{P}^2$ such that
$\sing(\SC)$ is finite --- and, hence, formed by a single point $p$. Then
the elements of  $\SC$ are contained in the pencil of lines with base point $p$.
The next proposition gives a version  of this fact for families of curves of higher degrees.
For its proof, we fix homogeneous $[z_{0}:z_{1}:z_{2}]$ in $\mathbb{P}^2$ and make the convention that the same Latin capital letter denotes both a   homogeneous polynomial in $\C[z_{0},z_{1},z_{2}]$ and  the algebraic curve that it defines.

\begin{prop}
\label{prop-infinite-collection}
Let $\SC$ be an infinite collection of irreducible curves in $\mathbb{P}^2$ of the same degree $d$. Suppose that $\sing (\SC)$ is a finite set.   Then there exist
 infinitely many curves in $\SC$ that are level sets of the same
rational function $F/G$ on $\mathbb{P}^2$.
 \end{prop}
\begin{proof}
 Let $d$ be  the common degree of the curves in $\SC$.  Fix a curve in $\SC$, which is defined by an irreducible homogeneous polynomial $P \in \C[z_{0},z_{1},z_{2}]$. Take a branch --- a local irreducible component --- $\gamma_1$ of $P$ at some
  $p_{1} \in \sing(\SC)$.
 By Bézout's Theorem, the intersection number at $p_{1}$  of $\gamma_1$  with the localization at $p_{1}$ of any other curve in $\SC$ is bounded by $d^2$. Then there exists an infinite set $\SC_1\subset \SC$ such that, at $p_{1}$,
 the intersection number
 $(\gamma_1,C)_{p_{1}}$ is the same
  for all $C\in \SC_1$. Take another branch $\gamma_2$  of $P$ at some
  point of $p_{2} \in \sing(\SC)$.  As before, there is an infinite set $\SC_2 \subset \SC_1$ such that
 $(\gamma_2,C)_{p_{2}}$ is the same number for all $C \in \SC_2$.
Proceeding  in this way, after a finite number of steps
  we   find  an infinite set $\SC^{*} \subset \SC$ such that, for each branch $\gamma$ of $P$ at
 any point of $p \in \sing(\SC)$, the intersection number $(\gamma,C)_{p}$ is the same for all
 $C \in \SC^{*}$. Let  $F$ and $G$ be  irreducible homogeneous polynomials defining
 two distinct curves in $\SC^{*}$. Take any curve in  $\SC^{*}$
 different from $F$ and $G$, say defined by an irreducible homogeneous polynomial $H$.
 Then it is enough to prove
 that $H$ belongs to the pencil
 $$\alpha F+\beta G, \quad [\alpha:\beta]\in \mathbb{P}^1.$$

 Let us show first that $P$ belongs to this pencil.
 Let $\gamma$ be a branch of $P$ at a point $p\in\sing(\SC)$. Let
 $f=0$ and $g=0$ be local equations of $F$ and $G$ at
  $\left(\mathbb{P}^2,p\right) \cong (\C^{2},0)$.
 Since $(\gamma,F)_{p}=(\gamma, G)_{p}=n$, for some fixed $n \geq 0$, if $\gamma(t)$  is a parametrization of the branch $\gamma$ we have
 $$f(\gamma(t))=at^n+O(t^{n+1}) \qquad \text{and} \qquad g(\gamma(t))=bt^n+O(t^{n+1}) $$
 for some constants
 $a,b\in\mathbb{C}^*$. Hence
 $$\alpha f(\gamma(t))+\beta g(\gamma(t))=(\alpha a +\beta b) t^n +O(t^{n+1}).$$
 From this we have   that
  $$(\gamma,\alpha F+\beta G)_{p} \ge n = (\gamma,F)_{p} , \quad \forall\ [\alpha:\beta]\in \mathbb{P}^1.$$
 However,  by choosing  $[\alpha:\beta]=[-b:a]$,  we have
  $$(\gamma,\alpha F+\beta G)_{p}> n=  (\gamma,F)_{p}.$$
  Now, we fix some branch $\gamma'$ of $P$ at some point  $q \in \sing(\SC)$ and choose
  $[\alpha'\colon\beta']$ such that
\begin{equation}
\label{eq-ineq-strict}
(\gamma',\alpha' F+\beta' G)_{q}> (\gamma',F)_{q}.
\end{equation}
As remarked above,  for any other branch $\gamma$ of $P$ at $q$, we have
\begin{equation}
\label{eq-ineq}
(\gamma,\alpha' F+\beta' G)_{q}\ge (\gamma,F)_{q}.
\end{equation}
Note that  $\sing(\SC)$ contains the intersection of $P$ with any curve in  $\SC^{*}$. 
In particular, it contains  the points of intersection of $P$ and $F$.
Therefore,
denoting by $\Gamma$  the set of branches of $P$
at the points of $\sing(\SC)$, we have
\begin{align}
\label{eq-intersection-F}
   (P,F)&= \sum\limits_{\gamma\in \Gamma,\, p \in \gamma}(\gamma,F)_{p}
  =(\gamma',F)_{q}+
  \sum\limits_{\gamma\in\Gamma \backslash\{\gamma'\}, \, p \in \gamma}(\gamma,F)_{p} .
\end{align}
On the other hand
\begin{align}
  \label{eq-intersection-pencil}
  (P,\alpha' F+\beta' G)&
  \geq (\gamma',\alpha' F+\beta' G)_{q}+
  \sum\limits_{\gamma\in\Gamma\backslash\{\gamma'\},\, p \in \gamma}(\gamma,\alpha' F+\beta' G)_{p},
\end{align}
the inequality arising from the fact that we are only considering branches of $P$ at points of $\sing(\SC)$.
Therefore, taking into account \eqref{eq-ineq-strict}, \eqref{eq-ineq}, \eqref{eq-intersection-F} and \eqref{eq-intersection-pencil}, we find
  \begin{align}\label{final} (P,\alpha' F+\beta' G) > (P,F)=d^2.
  \end{align}

 Now we show that $P$ has to coincide with the curve $\alpha' F+\beta' G$. To do this
 it is enough to show that some of the branches $\gamma\in \Gamma$ belongs to $\alpha' F+\beta' G$, since $P$ is irreducible and has the same degree of $\alpha' F+\beta' G$.
If it were not the case, all   indices $(\gamma,\alpha' F+\beta' G)_{p}$, for $\gamma\in \Gamma$,
would be finite and  \eqref{final} would be  a contradiction,
 because by Bézout's Theorem we should have
  \begin{align} (P,\alpha' F+\beta'G)=d^2.
  \end{align}
 We thus obtain that
\begin{equation}
\label{eq-pencil-FG}
P=\alpha'F+\beta'G.
\end{equation}

Repeating the same reasoning  above for the pair $\{F,H\}$ in place of $\{F,G\}$,
we find $[\mu:\nu] \in \mathbb{P}^{1}$  such that
\begin{equation}
\label{eq-pencil-FH}
P=\mu F+\nu H.
\end{equation}
Note that $\nu\neq 0$ because $P\neq F$.  Therefore, solving for $H$ in   \eqref{eq-pencil-FG} and \eqref{eq-pencil-FH},  we obtain
$$H=\frac{\alpha'-\mu}{\nu}F+\frac{\beta'}{\nu}G,$$
finishing the proof of the Lemma
\end{proof}

Proposition \ref{prop-infinite-collection} allows us to prove Theorem \ref{teo-extension-ndim}:

\begin{proof} [Proof {\rm (of Theorem \ref{teo-extension-ndim})}]
We first suppose that $M \subset \Pe^{2}$. In this case, $\sing(\xbar{M_{reg}})$ is a finite set.
Let $L\subset M$ be a leaf of the Levi foliation.  By hypothesis,
$\xbar{L}$ is an irreducible algebraic curve contained in $\xbar{M_{reg}}$. If $\SC$ denotes the collection of these curves, then
 $\sing(\SC)$ is contained in the finite set $\sing(\xbar{M_{reg}})$.
Since there are uncountably many Levi leaves, by reducing  $\SC$ if needed, we
can suppose that it is an infinite family formed by    irreducible algebraic curves of the same degree.
 By Proposition \ref{prop-infinite-collection}, possibly reducing $\SC$ again, we can assume that all these curves
are levels of a rational function $F/G$.
It only remains to show that the Levi-flat hypersurface $M$ is invariant by the foliation $\G$ defined by the first integral $F/G$.
Let $T(\G,M) \subset M$ be the set of tangency points between $\G$ and $M$, i.e. the subset  of $M$ formed by the points
$p \in   \left( M \setminus  M_{reg} \right) \cup \left(\sing(\G) \cap M_{reg} \right)$
along with all points
$p \in M_{reg} \setminus \sing(\G)$ such that $T_{p} \G \subset T_{p} M$.
We have that $T(\G,M)$ is  real-analytic. Since $\SC \subset T(\G,M)$ and $\SC$ contains infinitely many distinct  curves, we must have
$\dim_{\R} T(\G,M) = 3$. On the other hand, since also $\dim_{\R} M = 3$ and $M$ is irreducible, we conclude that $T(\G,M) = M$, finishing the Theorem's proof
in the two dimensional case.

Suppose now that $M \subset \Pe^{n}$, with $n \geq 3$.
 As in the case $n=2$, we can take an infinity family $\SC$ of irreducible algebraic hypersurfaces of the same degree contained in $\xbar{M_{reg}}$.
Fix a pair of hypersurfaces of $\SC$, denoted by $F$ and $G$, the same notation of their equations in homogeneous coordinates of $\mathbb{P}^n$.
Let $H$ be an equation of
some hypersurface in $\SC$.
 We consider   two dimensional planes $\Pi  \simeq \mathbb{P}^2  \subset \mathbb{P}^n$ in \emph{general position} with $M$ and $\mathcal{A} = \{F,G,H\}$.
 By this, we mean the following:
\begin{enumerate}
\item $M' = M|_{\Pi}$ is a real analytic Levi-flat hypersurface such that $\sing(\xbar{M_{reg}'})$ has isolated singularities;
\item  $\mathcal{A}' = \{F',G',H'\}$ is a set of irreducible curves of the same degree, where $F'$ is the restriction  of $F$ to  $\Pi$, with similar definitions for $G'$ and $H'$.
\end{enumerate}
Since $\dim_{\R} \sing(\xbar{M_{reg}}) \leq 2n-4$, the generic element in the Grasmannian $G(2,n)$ satisfies  condition (1), and the same is true for condition (2).
Fix $\Pi  \simeq \mathbb{P}^2$  one such  plane.
By
Proposition \ref{prop-infinite-collection}, $H'$ is in the pencil generated by $F'$ and $G'$
(the rational function $F'/G'$ is a first integral for the foliation on $\Pi$  that extends the Levi foliation of $M'$).
Let us consider on $\mathbb{P}^n$ the pencil generated by $F$ and $G$. Since, for   generic   $\Pi$,  $H'$ is in the pencil generated by $F'$ and $G'$,  we have that $H$ is in the pencil generated by $F$ and $G$. Repeating this argument for every $H$ in $\SC$,
we conclude that all elements of $\SC$ belong to the pencil  generated by $F$ and $G$.
In order to conclude the proof, we have to show that the holomorphic foliation $\G$ given by the levels of $F/G$ extends the Levi foliation of $M$.
Indeed, as in the first part of the proof,  it is enough to consider that  $T(\G,M)$,
the set of points of tangency between $\G$ and $M$, is an analytic subset of $M$ that contains $\SC$.
Thus,  $\dim_{\R} T(\F,M) = 2n -1$ and, since
 $M$ is irreducible, we must have $T(\G,M) = M$,
completing the proof of the theorem.
\end{proof}

If $M \subset \Pe^{n}$ is a  real algebraic Levi-flat hypersurface, then all leaves of its Levi foliation are algebraic. Indeed, in the complex cone
$\C^{n} \setminus \{0\}$ of $\Pe^{n}$,  the lifting $\tilde{M}$ of $M$ is defined by an  equation of the form $\phi(z,\bar{z}) = 0$, where $\phi(z,\bar{z})$ is bihomogeneous
polynomial of  bidegree $(d,d)$ in the variables $(z,\bar{z})$, where $d>0$. If $z_{0} \in  \C^{n} \setminus \{0\}$
is a regular point of $\tilde{M}$, then the leaf of $\tilde{\LL}$ passing through $z_{0}$ is contained in the Segre variety $\Sigma_{z_{0}}$, whose equation
is $\phi(z,\bar{z}_{0}) = 0$. Since $z_{0}$ is not Segre degenerate, we conclude that $L_{z_{0}}$ is algebraic.
In \cite[Theorem 1.2 item (iii)]{jiri2012}, it is proved that if   $M \subset \Pe^{n}$ is a real algebraic  Levi-flat hypersurface with $\dim_{\R} \sing(\xbar{M_{reg}}) = 2n -4$,
then there is a complex subvariety $S \subset M$ of complex dimension $n-2$ whose points are Segre degenerate points of $M$.
As an application of Theorem \ref{teo-extension-ndim}, we have an improvement  of this result, by replacing the hypothesis on the algebraicity of $M$ to
that of its Levi leaves:

\begin{cor}
Let $M$ be a real analytic Levi-flat hypersurface  in
$\mathbb{P}^n$, $n \geq 2$, such that $\dim_{\R} \sing(\xbar{M_{reg}}) = 2n -4$.  Suppose that the Levi leaves
of $M$ are all algebraic. Then the set of Segre degenerate singularities of $M$ form a complex algebraic variety  of pure dimension   $n-2$, contained in   $\xbar{M_{reg}}$.
\end{cor}
\begin{proof}
Applying Theorem \ref{teo-extension-ndim}, we have that $M$ is tangent to a  holomorphic foliation $\G$ on $\mathbb{P}^n$ with a rational first integral.
By \cite[Th. 3.1]{Pinchuk_2017},
Segre degenerate points are contained in the indeterminacy set of this rational function, which is an algebraic set of pure codimension two.
For  $n\geq 3$,   Proposition \ref{prop-segredegenerate-codimtwo} gives that this set has some of its  components in $M$.
For $n=2$, it is enough  to consider that
the hypothesis on the dimension of $\sing(\xbar{M_{reg}})$ gives that this  is a finite set and, since there are infinitely many algebraic Levi leaves,
  infinitely many of them accumulate to a same point $p \in \sing(\xbar{M_{reg}})$, which turns out to be a Segre degenerate singularity.
\end{proof}

\begin{example}
\label{ex-lebl}
{\rm  (J. Lebl, \cite{lebl2015})  Take $X \subset \R^{2}$ a smooth real analytic curve that is not contained
in any proper real algebraic variety of $\R^{2}$. The set
\begin{equation}
\label{eq-ex-lebl}
\tilde{M} = \{(z_{0},z_{1},z_{2}) \in \C^{3} \setminus \{0\}; \ z_{0} = x z_{1} + yz_{2} \ \ \text{with} \ \ (x,y) \in X\} \cup \Delta ,
\end{equation}
where $\Delta = \{(z_{0},z_{1},z_{2})  \in \C^{3}; z_{1} \bar{z}_{2} = \bar{z}_{1} z_{2} \}$,
is a real analytic hypersurface of $\C^{3} \setminus \{0\}$.
The fact that $X$ is non-algebraic implies that any polynomial vanishing on $\tilde{M}$ must be
identically zero.    $\tilde{M}$ is a complex cone with vertex $0 \in \C^{3}$. It is also Levi-flat, foliated by complex planes. Thus, its image
  by the canonical projection $\C^{3} \setminus \{0\} \to \Pe^{2}$ is a non-algebraic,  real analytic hypersurface, $M \subset \Pe^2$, foliated by projective lines.
A Chow type theorem, along the lines of  our main theorem, fails in this case. The singular set $\sing(\tilde{M})$ coincides with the intersection of the first
set in \eqref{eq-ex-lebl}
with $\Delta$, which is a set of real dimension four. Thus, $\dim_{\R}\sing(M)=2$ and $\sing(M)$ is non-discrete.
}\end{example}

\begin{example}
\label{ex-brunella}
{\rm  (M. Brunella, \cite{brunella2007})
Let $(z_{1},z_{2})$ be the coordinates of $\mathbb{C}^{2}$
such that $z_{1}=x+iy$, $z_{2}=s+it$ and take the  hypersurface $M$ defined by
\begin{equation}
M=\{(z_{1},z_{2})\in\mathbb{C}^{2}:t^{2}=4(y^{2}+s)y^{2}\}.
\end{equation}
Considering $(z_{1},z_{2})$ as affine coordinates of $\Pe^2$, we can view $M$ as a real algebraic  hypersurface in  $\mathbb{P}^2$. Moreover,
$M$ is Levi-flat: it is foliated by  the complex curves
\[L_c=\{z_{2}=(z_{1}+c)^2: \IM(z_{1})\neq 0\},\,\,\,\,\,c\in\mathbb{R}.\]
Furthermore,  $\sing(M)=\{t=y=0\}$, so that $\dim_{\R}\sing(M)=2$. In this case, even though $M$ is algebraic, the Levi foliation does not extend to a holomorphic foliation on the ambient
$\Pe^{2}$.  However, there exists a $2$-web $\mathcal{W}$ extending it, namely the one defined by the equation $(dw)^{2}-4w(dz)^{2}=0$
(see \cite{pereira2015}).
Once again, this example shows  how the
hypothesis on the dimension of the singular set is essential in order to achieve the conclusion of our main theorem.
}\end{example}

\section{Levi foliations}
\label{section-Levi-flat-foliations}

Let $X$ be a holomorphic manifold of dimension $n$. A  real analytic foliation $\G$  on $X$ of real codimension two  is a \emph{Levi foliation} if its leaves are
holomorphic manifolds of complex codimension one immersed in $X$.  More precisely, $\G$ is defined by
 a system of  real analytic local charts $(\phi,U)$, where $\phi:U \to \Delta$ is a real analytic diffeomorphism, $U \subset X$ being an open set and $\Delta \subset \C^{n}$, the unitary polydisc centered at the origin,
such that, for overlapping charts $(\phi_{\alpha},U_{\alpha}), (\phi_{\beta},U_{\beta})$, the real analytic transition functions $\phi_{\alpha \beta} = \phi_{\beta} \circ \phi_{\alpha}^{-1}$ have the form $\phi_{\alpha \beta}(Z',z_{n}) = (f_{\alpha \beta}(Z',z_{n}), g_{\alpha \beta}(z_{n})) \in \C^{n-1} \times \C$, with
 $f_{\alpha \beta}$
holomorphic in the variables $Z' = (z_{1},\ldots,z_{n-1}) \in \C^{n-1}$. In this way, for each chart $(\phi,U)$ and  $c \in \mathbb{D}$, the sets $\Delta \cap \{z_{n}=c \}$
are plates that glue together  forming the leaves of $\G$.

Locally on $X$, the distribution of complex hyperplanes in the holomorphic cotangent bundle $T^{*}X$ induced by $\G$ is defined by a non-singular real-analytic section, that is, by a  real analytic $1$-form of type $(1,0)$ without zeroes. For instance, in the chart $(\phi,U)$, we can take $\eta = \phi^{*} dz_{n}$.
For overlapping charts $(\phi_{\alpha},U_{\alpha}), (\phi_{\beta},U_{\beta})$, the corresponding $1$-forms $\eta_{\alpha}$ and $\eta_{\beta}$ are such that
$\eta_{\alpha} = \lambda_{\alpha \beta} \eta_{\beta}$ for some real analytic function $\lambda_{\alpha \beta}: U_{\alpha} \cap U_{\beta} \to \C^{*}$.

We say that $\G$ is  a \emph{Levi foliation with singularities}  on $X$  if there is a covering $\{U_{\alpha}\}$ of $X$ by connected open sets and, for each $\alpha$,
there is a non-zero real analytic $1$-form $\eta_{\alpha}$ of type $(1,0)$ that defines a (regular) Levi-flat foliation on $U_{\alpha} \setminus \sing(\eta_{\alpha})$ and such that, whenever $U_{\alpha} \cap U_{\beta} \neq \emptyset$, there exists a real analytic $\lambda_{\alpha \beta}: U_{\alpha} \cap U_{\beta} \to \C^{*}$ such that
$\eta_{\alpha} = \lambda_{\alpha \beta} \eta_{\beta}$ in $U_{\alpha} \cap U_{\beta}$.
The singular set of $\G$ is the real analytic set $\sing(\G) \subset X$  defined by $\sing(\G)|_{U_{\alpha}} = \sing(\eta_{\alpha})$.

We provide the following consequence of Proposition \ref{prop-infinite-collection}, that assures the existence of
semialgebraic Levi-flat hypersurfaces tangent to real analytic Levi foliations in the complex projective plane  $\Pe^{2}$:

\begin{prop}
\label{prop-invariant-surface}
 Let $\G$ be a real analytic Levi foliation on $\Pe^{2}$ with isolated singularities. If it
  contains uncountably many algebraic leaves, then there exists a semialgebraic hypersurface invariant by $\G$.
\end{prop}
\begin{proof}
Within the family of  algebraic leaves of $\G$, there exists a subfamily $\SC$ formed by infinitely many algebraic curves having the same degree.
The singular set $\sing(\SC)$ is contained in $\sing(\G)$, hence it is discrete.
By Proposition \ref{prop-infinite-collection}, infinitely many curves of $\SC$ are contained in the levels of a rational function $F/G$,
where $F,G \in \C[z_{0},z_{1},z_{2}]$. Denote by $\mathcal{R}$ the foliation defined by the levels of  $F/G$ and by
$T(\G,\mathcal{R})$   the set of tangencies of $\G$ and $\mathcal{R}$.  Evidently, $T(\G,\mathcal{R})$  is a real analytic set that contains
 $\SC$. Then, either $T(\G,\mathcal{R})$ is the whole $\Pe^{2}$  or it is a real analytic hypersurface $H$.  In the first case, $\G$ and $\mathcal{R}$ coincide and, evidently, there are plenty of real algebraic hypersurfaces invariant by $\G$. In the second case, $H$ is a real analytic Levi-flat hypersurface with infinitely many
algebraic leaves. By \cite[Th. 6.6]{jiri2012} it is semialgebraic.
\end{proof}

We close this article giving a version of Theorem  \ref{teo-extension-ndim} for singular Levi-foliations in $\Pe^{n}$:
\begin{teo}
\label{teo-levi-foliation}
 Let $\G$ be a real analytic Levi foliation on $\Pe^{n}$, $n \geq 2$,
 such that $\sing(\G)$ has real dimension at most $2n-4$. If
all the  leaves $\G$ are algebraic, then  it has a rational first integral.
\end{teo}
\begin{proof}
Let us first prove the result for $n=2$. In this case, $\G$ has isolated singularities and all its leaves are algebraic curves.
 Denote   by $\SC$ the family of algebraic leaves of $\G$  and by $\SC_{d}$ the subfamily of those of degree  $d$,  for $d \geq 1$.
Let us suppose, for the moment, that $d$ is fixed.  We fix a curve in $\SC_{d}$, denoted by its
irreducible homogeneous equation  $P \in \C[z_{0},z_{1},z_{2}]$. Let $\gamma_{1},\ldots, \gamma_{\ell}$ denote the family of local branches of $P$ at points
of $\sing(\SC) \subset \sing(\G)$, say $\gamma_{i}$ is a local branch at $p_{i} \in \sing(\SC)$.
It follows from the proof of Proposition \ref{prop-infinite-collection} that,
 if $\SC_{d}^{*} \subset \SC_{d}$ is an infinite family of curves such that $(\gamma_{i},C)_{p_{i}}$ is the same for every $i=1,\ldots,\ell$ and every $C \in \SC_{d}^{*}$, then all curves of $\SC_{d}^{*}$ are contained in a pencil, say
the one defined by the rational function $F/G$, where $F,G \in \C[z_{0},z_{1},z_{2}]$ are homogeneous equations of two elements in $\SC_{d}^{*}$, which defines
a holomorphic foliation $\mathcal{R}$ on $\Pe^{2}$. Thus,
as we have seen in the proof of Proposition \ref{prop-invariant-surface}, either the tangency set
$T(\G,\mathcal{R})$ is the whole $\Pe^{2}$  or it is a semialgebraic  hypersurface $M$.
Let us consider $\ell$-uples $\iota \in \mathbb{Z}_{\geq 0}^{\ell}$ and, for
  $C \in \SC_{d}$, with $C \neq P$, define $\iota(C) =  \left( (\gamma_{1},C)_{p_{1}},\ldots, (\gamma_{\ell},C)_{p_{\ell}} \right)$.
For a fixed $\iota  \in \mathbb{Z}_{\geq 0}^{\ell}$,
let us denote  $\SC_{d,\iota} = \{ C \in \SC_{d}; \iota(C) = \iota \}$.
Denote by $\Upsilon_{\infty}$ the set of multi-indices $\iota$ for which $\SC_{d,\iota}$ is infinite and by $\Upsilon_{0}$ its
complementary set in $\mathbb{Z}_{\geq 0}^{\ell}$.
In order to get a contradiction, following our previous discussion,  we suppose that, whenever $\iota \in \Upsilon_{\infty}$,  all
 curves in $\SC_{d,\iota}$ are contained in a semialgebraic Levi-flat hypersurface $M_{\iota}$ (and that this holds for every degree $d \geq 1$).
For $\iota \in \Upsilon_{0}$, let us denote by $N_{\iota}$ the finite union of curves in  $\SC_{d,\iota}$.
Denoting $\cl{A}_{d} = \cup_{C \in \SC_{d}}$ and $\cl{A}_{d,\iota} = \cup_{C \in \SC_{d,\iota}}$, we have
\[\cl{A}_{d} = \bigcup_{\iota  \in \mathbb{Z}_{\geq 0}^{\ell}} \cl{A}_{d,\iota} =   \left( \bigcup_{\iota \in \Upsilon_{\infty}} M_{\iota} \right) \bigcup \left(   \bigcup_{\iota \in \Upsilon_{0}} N_{\iota} \right).\]
Since each $M_{\iota}$ and each  $N_{\iota}$ have zero Lebesgue measure in $\Pe^{2}$ and the relation above involves countable unions, we
have that $\cl{A}_{d}$ has zero  measure in $\Pe^{2}$.
We finish by noting that, on the one hand $\cup_{d \geq 1} \cl{A}_{d}$ equals $\Pe^2$ (perhaps minus a finite number of points in $\sing(\G)$).
On the other hand, it is a countable union of sets of Lebesgue measure zero. This gives a contradiction that proves the theorem in dimension two.

The passage from dimension two to an arbitrary dimension $n>2$ is done by means of two dimensional linear sections, employng the
same arguments of the proof of Theorem \ref{teo-extension-ndim}. We leave the details for the reader.
\end{proof}

The following example illustrates Theorem \ref{teo-levi-foliation}:
\begin{example}
{\rm (J-M. Lion, cited in \cite[Sec.2]{cerveau2004})
Let $A:\C \to \C$ be an $\R$-linear isomorphism that is not $\C$-linear.
For each $z \in \C$, we take the line joining the points $(0,z)$ and $(1,Az)$.
This produces a distribution of  complex lines in $\C^{2}$, which defines a real analytic Levi foliation $\G$ on $\Pe^{2}$.
This foliation is not holomorphic, since the holonomy map between the sections $\Sigma_{0}: z_{1} = 0$ and $\Sigma_{1}: z_{1} = 1$
is not holomorphic.

The line $L_{z}$ of $\G$ passing through $(0,z) \in \Sigma_{0}$ is parameterized as
\[ (0,z) + t(1,A(z) - z) = (t, z + t(A - I)(z)), \ \ \ t \in \C,\]
where $I$ denotes the identity map in $\C$.
Suppose that, for distinct $z,\tilde{z} \in \C$,  there is a point of intersection of $L_{z }$ and $L_{\tilde{z} }$.
We find this point
 by
solving the following equation for $t , \tilde{t} \in \C$:
\[ (t , z  + t (A - I)(z )) = (\tilde{t} ,\tilde{z} + \tilde{t}(A - I)(\tilde{z})) .\]
Thus, $t  = \tilde{t}$ and
\[ z  + t(A - I)(z) =  \tilde{z}  + t(A - I)(\tilde{z}) ,\]
which gives
\[ t(A - I)(\tilde{z}-z)  = -(\tilde{z}-z) \Rightarrow  (A - I)(\tilde{z}-z)  = -\frac{1}{t}(\tilde{z}-z). \]
That is, $w = \tilde{z}-z$ must be an eigenvector of $A-I$ of eigenvalue $\lambda = -1/t \neq 0$.
Thus, for every $\alpha  \in \R \setminus \{0\}$, $L_{z }$ and $L_{z + \alpha w}$ have a common point of intersection
on the line $z_{1} = -1/\lambda$. This point belongs to $\sing(\G)$.
Each   real  line  on the complex line $\Sigma_{0}$ parallel to the eigenvector $w$ determines different point of $\sing(\G)$ on the line $z_{1} = -1/\lambda$.
Thus, we see that $\sing(\G)$ is not discrete (this too shows that $\G$ is not a singular holomorphic foliation).
We thus have an example where the conclusion of  Theorem  \ref{teo-levi-foliation} fails when the hypothesis on the
dimension of the singular set is dropped.

}\end{example}

\bibliographystyle{plain}
\bibliography{referencias}

%
%
%
%
%
%

\end{document}